\documentclass[12pt]{amsart}
\usepackage[T1]{fontenc}
\usepackage[utf8]{inputenc}
\usepackage{fourier}
\usepackage[active]{srcltx}
\usepackage{a4wide}
\usepackage{amsthm,amsfonts,amsmath,mathrsfs,amssymb}
\usepackage{dsfont}
\def\be{\begin{equation}}
\def\ee{\end{equation}}

\newtheorem*{completeness*}{Completeness property}
\newtheorem*{theorem*}{Theorem}
\newtheorem{theorem}{Theorem}

\newtheorem*{proposition*}{Proposition}
\newtheorem{lemma}{Lemma}
\newtheorem{corollary}{Corollary}
\theoremstyle{remark}

\newcommand{\nc}{\newcommand}

\newcommand{\C}{{\mathbb C}}
\newcommand{\N}{{\mathbb N}}
\newcommand{\D}{{\mathbb D}}

\nc{\supp}{\operatorname{supp}}
\nc{\dif}{\operatorname{d}} \nc{\im}{\operatorname{i}}
\nc{\Hi}{{\mathscr{H}}^\infty} \nc{\Ht}{{\mathscr{H}}^2}
\nc{\Hone}{{\mathscr{H}}^1} \nc{\ol}{\overline} \nc{\bz}{\mathbf{z}}
\nc{\bw}{\mathbf{w}} \nc{\eps}{\varepsilon}

\allowdisplaybreaks[3]

\begin{document}
\title[GCD sums and complete sets of square-free numbers]{GCD sums and complete sets of square-free numbers}
\author{Andriy Bondarenko}
\address{Department of Mathematical Analysis\\ Taras Shevchenko National University of Kyiv\\
Volody- myrska 64\\ 01033 Kyiv\\ Ukraine}
\address{Department of Mathematical Sciences \\ Norwegian University of Science and Technology \\ NO-7491 Trondheim \\ Norway}

\email{andriybond@gmail.com}
\author[Kristian Seip]{Kristian Seip}
\address{Department of Mathematical Sciences \\ Norwegian University of Science and Technology \\ NO-7491 Trondheim \\ Norway}
\email{seip@math.ntnu.no}
\thanks{Research supported by Grant 227768 of the Research Council of Norway. The first author is a fellow of the European Research Consortium for Informatics and Mathematics (ERCIM) ``Alain Bensoussan'' Fellowship Programme.}
\maketitle

\begin{abstract}
It is proved that
\[
\sum_{k,{\ell}=1}^N\frac{\gcd(n_k,n_{\ell})}{\sqrt{n_k
n_{\ell}}} \ll N\exp\left(C\sqrt{\frac{\log N \log\log\log N}{\log\log N}}\right)
\]
holds for arbitrary integers $1\le n_1<\cdots < n_N$. This bound is essentially better than that found in a recent paper of Aistleitner, Berkes, and Seip and can be improved by no more than removal of the triple logarithm. A certain completeness property of extremal sets of square-free numbers plays an important role in the proof of this result.    \end{abstract}

\section{Introduction}

The study of greatest common divisor (GCD) sums of the form
\begin{equation}\label{gcda}
\sum_{k,\ell=1}^N\frac{(\gcd(n_k,n_{\ell}))^{2\alpha}}{(n_k
n_{\ell})^\alpha}
\end{equation}
%began with a paper of G\'{a}l \cite{G}, solving 
was initiated by Erd\H{o}s who inspired G\'{a}l \cite{G} to solve 
a prize problem of the Wiskundig Genoot- schap in Amsterdam concerning the case $\alpha=1$. 
G\'{a}l proved that when $\alpha=1$, $CN(\log \log N)^2$ is an optimal upper bound for \eqref{gcda}, with $C$ an absolute constant independent of $N$ and the distinct positive integers $n_1,...,n_N$. The problem solved by G\'{a}l had been posed by Koksma in the 1930s, based on the observation that such bounds would have implications for the uniform distribution of sequences $(n_k x)$ mod 1 for almost all $x$ \cite{K}. 

The purpose of the present paper is to obtain the following new bound\footnote{Here and in what follows, we assume that $N\ge e^3$ so that $\log\log\log N$ is well defined and positive. We also assume that the implied constant in any relation $L\ll R$ is absolute.} for \eqref{gcda} when $\alpha=1/2$. 
\begin{theorem} \label{maintheorem} There exists an absolute constant $A$ less than 7 such that
\begin{equation}\label{gcdb}
\sum_{k,\ell=1}^N\frac{\gcd(n_k,n_{\ell})}{\sqrt{n_k
n_{\ell}}} \ll N \exp\left(A\sqrt{\frac{\log N \log\log\log N}{\log\log N}}\right)
\end{equation}
for arbitrary integers $1\le n_1<\cdots <n_N$.
\end{theorem}
The first estimate in this case was found by Dyer and Harman in 1986 \cite{DH}, showing that the sum in \eqref{gcdb} is bounded by $\exp(C \log N/\log\log N)$.  The better bound $ \exp(C\sqrt{\log N \log\log N})$ was found by Aistleitner, Berkes, and Seip in  \cite{ABS} by a method that also led to a full solution when $1/2<\alpha <1$ (see Section~\ref{concluding} for details). The research carried out in  \cite{ABS} made it clear that  $\alpha=1/2$ represents a profound limiting case that deserves separate attention and requires novel techniques. This recognition led to the present investigation. 

In spite of our efforts to obtain an optimal bound, there is still room for a slight improvement, as shown by the following result from \cite{ABS}: For large $N$ there exist integers $n_1<\cdots <n_N$ such that
\be \label{lower}   \sum_{k,\ell=1}^N\frac{\gcd(n_k,n_{\ell})}{\sqrt{n_k
n_{\ell}}} \ge N \exp\left(c\sqrt{\frac{\log N}{\log\log N}}\right) \ee
for an absolute constant $c>0$ \cite{ABS}. Thus it remains to decide whether the factor $\sqrt{\log\log\log N}$ in the exponent on the right-hand side of \eqref{gcdb} can be removed  or replaced by a function growing even more slowly. 

The problem of estimating the sum in \eqref{gcdb} arose in the metric theory of diophantine approximation, with the first variant of it appearing in the work of LeVeque \cite{V}. Several applications of the estimate from \cite{DH} were considered by Harman, cf. \cite{harmand}, his work on the Duffin--Schaffer conjecture \cite{harmans}, and Chapter~3 of his book \cite{H}. The study of GCD sums of the form \eqref{gcdb} can also be seen to originate in Wintner's seminal work on systems of dilated functions \cite{W}. See \cite{bewe} and in particular \cite{ABS}, where bounds for \eqref{gcdb} in the range $1/2<\alpha<1$ were used to establish a Carleson--Hunt inequality that in turn settled two long-standing problems regarding the almost everywhere behavior of systems of dilated functions. (We note in passing that the problem of estimating  \eqref{gcda} for $1/2<\alpha<1$ was first stated by Mikol\'{a}s in connection with a problem involving the Hurwitz zeta-function \cite{M}.) Whether our bounds for $\alpha=1/2$ could have applications in this field, for instance to Khinchin's conjecture \cite{BW}, remains an open question. 

Applications usually require bounds for the spectral norm (or largest eigenvalue) of  the corresponding GCD matrix  
$((\gcd(n_k,n_{\ell}))^{2\alpha}/(n_k
n_{\ell})^\alpha)$. Such bounds follow easily from estimates of GCD sums when $0<\alpha <1$; we will state a precise result of this kind in Section~\ref{concluding}. However, as in \cite{ABS}, our method of proof involves a certain combinatorial argument that relies crucially on the special structure of the sum in \eqref{gcdb}. We therefore consider the GCD sums rather than the corresponding matrices as our primary object of study. 
 
In \cite{Hi}, Hilberdink estimated the spectral norms in the distinguished case when the GCD matrix is generated by the first $N$ integers, yielding lower bounds for the growth of the Riemann zeta-function $\zeta(s)$ on vertical lines. As regards the dependence on $\alpha$ in our results, one may observe a curious similarity with the anticipated growth of $\zeta(s)$ along vertical lines $\sigma=\alpha$. In fact, in the trivial case $\alpha>1$, it can be shown that $(\zeta(\alpha))^2/\zeta(2\alpha)$ is the supremum of the spectral norms of our GCD matrices, cf. Theorem~6 in \cite{ABS} and \cite{LS}. It seems unlikely, though, that there is a similarly explicit relation for $1/2\le \alpha\le 1$. We have rather come to regard our  problem as a more accessible relative to the problem of estimating the growth of $|\zeta(\alpha+it)|$.

To close this introduction, we give a brief outline of the paper. The proof of Theorem~\ref{maintheorem} will be presented in Section~\ref{proof}. We will obtain Theorem~\ref{maintheorem} essentially as a consequence of the bound we get when requiring the numbers $n_1,..., n_N$ to be square-free. The square-free case will be dealt with in Sections  \ref{extremal}--\ref{proofsfc}, after the next section, in which we will present  the notion of completeness referred to in the title of this paper. Some concluding remarks are gathered in the final Section~\ref{concluding}.

\section{Extremal sets of square-free numbers}\label{first}

We will use the multi-index notation introduced in \cite{ABS}. This means that we will represent a positive integer as the multi-index corresponding to its prime factorization and a completely multiplicative function bounded by one as a point in the infinite-dimensional polydisc. However, before introducing this notation, we will set the stage by outlining, in basic number theoretic terms, the main conceptual novelty of our work.

We fix $\alpha>0$ and associate with every set of distinct integers $\mathcal{N}:=\{n_1,...,n_N\}$ its GCD sum \eqref{gcda}. 
We confine ourselves for the moment to the case when the numbers $n_1,..., n_N$ are all  square-free. G\'{a}l showed that, under this restriction,
it may be further assumed that $\mathcal{N}$ is divisor closed, meaning that if $n$ is in $\mathcal{N}$, then all its divisors are also in $\mathcal{N}$. More precisely, G\'{a}l found a simple division algorithm allowing him to transform any given set $\mathcal{N}$ of square-free numbers into another set $\mathcal{N}'$ of square-free numbers such that the latter set is divisor closed and its GCD sum is larger than or equal to that of $\mathcal{N}$. This basic fact played a crucial role in  G\'{a}l's proof that the optimal bound in the square-free case (when $\alpha=1$) is $CN\log \log N$. We will see below that it also implies the existence of extremal sets of square-free numbers, i.e., sets having maximal GCD sums among all sets of the same cardinality $N$.

G\'{a}l's result gives important information about extremal sets of square-free numbers, but the structure of such sets may still be rather complicated. What our work has led us to, is another basic property of extremal sets of square-free numbers: An extremal divisor closed set of square-free numbers has the following property. 
\begin{completeness*}
A set $\mathcal{N}$ of distinct positive integers has the completeness property if for every $n$ in $\mathcal{N}$ and prime $p$ such that $p|n$, we have that either $p'|n$ or $np^{-1}p'$ is in $\mathcal{N}$ whenever $p'$ is a prime smaller than $p$ . 
\end{completeness*}
The completeness property reflects that integers in extremal sets of square-free numbers should be composed of as small primes as possible. It turns out that this ``domination'' of small primes can be quantified and used to improve the method of estimation from \cite{ABS}  significantly.

A set of positive integers will be said to be complete if it is divisor closed and enjoys the completeness property. It follows from what was said above that, if we restrict to the case of square-free numbers, then we may assume that the set $\mathcal{N}$ is complete.
  
The general case, without the assumption that the numbers in $\mathcal{N}$ are square-free, is more difficult. It is clear from G\'{a}l's work that his division algorithm does not work, so that we cannot simply restrict to divisor closed sets. It is likewise clear that we cannot take for granted that our sets have the completeness property, and  it is not even clear whether extremal sets exist. However, combining a basic lemma from \cite{ABS} with an estimate for the spectral norm of the corresponding GCD matrix, we are able to transfer our analysis of the general case to that of square-free numbers. Thus complete extremal sets of square-free numbers are indeed crucial also in the general case.

\section{Statement of the main result in the square-free case}\label{extremal}

Following \cite{ABS}, we declare a multi-index to be a sequence $\beta=
(\beta^{(1)}, \beta^{(2)}, ..., \beta^{(R)}, 0, 0, ...)$ consisting of nonnegative integers with only a finite number of them being nonzero. We let $\supp \beta$ be the finite set of positive integers $j$ for which $\beta^{(j)}>0$.
% we write $R(\beta)$ for the maximal element in $\supp \beta$. 
Two multi-indices $\beta$ and $\mu$ may be added and subtracted as sequences. Then $\beta-\mu$ may fail to be a multi-index, but the sequence $|\beta-\mu|=(|\beta^{(j)}-\mu^{(j)}|)$ will again be a multi-index. We may multiply multi-indices by positive integers in the obvious way and express any multi-index as a linear combination of the natural basis elements $e_j$, where $e_j$ is the multi-index supported by $\{j\}$ with $e^{(j)}_j=1$.  We write
$\beta\le \mu$ if $\beta^{(j)}\le \mu^{(j)}$ for every $j$. For a sequence of complex numbers $z=(z_j)$, we  use the notation
\[ {z}^\beta:=\prod_{j\in \supp \beta} z_j^{\beta^{(j)}};\]
we will sometimes write $z^{-\beta}$ for the number $(z^{\beta})^{-1}$. A sequence $z$ belongs to $\D^{\infty}$ if $|z_j|<1$ for every $j$, and it belongs to $c_0$ if $z_j\to 0$ when $j\to \infty$.

We write $p=(p_j)$ for the sequence of prime numbers ordered by ascending magnitude. Using our multi-index notation, we may write every positive integer $n$ as $p^{\beta}$ for a multi-index $\beta$ that is uniquely determined by $n$. If $n_k=p^{\beta_k}$, then we may write
\[   \frac{(\operatorname{gcd}(n_k,n_{\ell}))^{2}}{n_k
n_{\ell}} = p^{-|\beta_k-\beta_{\ell} |}.\]
For an arbitrary positive sequence $t$ in $\D^{\infty}$ and a set of distinct multi-indices $B=\{\beta_1,..., \beta_N\}$, we now define
\[ S(t,B):=\sum_{k,{\ell}=1}^N t^{|\beta_k-\beta_{\ell} |}. \]
We set
\[ \Gamma_t(N):=N^{-1}\sup_{B } S(t,B),\]
where the supremum is taken over all possible sets $B$ of distinct multi-indices $\beta_1,...,\beta_N$. Our original problem has thus been transformed into the problem of estimating $\Gamma_t(N)$ in the particular case when $t=(p_j^{-1/2})$.

A multi-index $\beta$ will be said to be square-free if $\beta^{(j)}\le 1$ for every $j$; a set $B$ of multi-indices will correspondingly be said to be square-free if each $\beta$ in $B$ is square-free. We set
\[ \Gamma^{\text{sf}}_t(N):=N^{-1}\sup_{B\ \text{square-free} } S(t,B),\]
where it is understood that all $B$ consist of $N$ distinct square-free multi-indices.

Our main result in the square-free case can now be expressed as follows; here and in the sequel we use the short-hand notation $\log_2 N:=\log\log N$ and $\log_3 N:=\log\log\log N$. 
\begin{theorem} Let $t=(t_j)$ be a decreasing positive sequence in $\D^{\infty}$ satisfying the estimate $t_j\le C/\sqrt{j \log j}$ for sufficiently large  $j$. Then
\begin{equation}\label{est2} \Gamma^{\operatorname{sf}}_t(N)\ll \exp\left(\kappa \sqrt{C}\sqrt{\frac{\log N \log_3 N}{\log_2 N}}\right),   \end{equation}
where $\kappa$ is an absolute constant less than $5$.
\label{sfc}
\end{theorem}

In view of Section~4 of \cite{ABS}, this theorem implies a similar estimate for the largest eigenvalues of the GCD matrices $(t^{|\beta_k-\beta_{\ell}|})$. As can be seen from the Poisson integral formula of \cite{ABS},  these matrices are positive definite. This means that
\[ \Lambda^{\operatorname{sf}}_t(N):=\sup_{\beta_1,...,\beta_N\ \text{square-free}} \sup_{c\neq 0} \frac{\sum_{k,\ell=1}^N t^{|\beta_k-\beta_{\ell}|} c_k \overline{c}_{\ell}}{\sum_{j=1}^N |c_j|^2} \]
is the least upper bound for these eigenvalues, where the suprema are taken over respectively all $N$-tuples of distinct square-free multi-indices $\beta_1,...\beta_N$ and all nonzero vectors $c=(c_1,...,c_N)$ in $\C^N$.   

\begin{corollary} Let $t=(t_j)$ be a decreasing positive sequence in $\D^{\infty}$ satisfying the estimate $t_j\le C/\sqrt{j \log j}$ for sufficiently large  $j$. Then
\[ \Lambda^{\operatorname{sf}}_t(N)\ll \exp\left(\kappa \sqrt{C}\sqrt{\frac{\log N \log_3 N}{\log_2 N}}\right),\]
where $\kappa$ is an absolute constant less than $5$.
\label{sfceigen}
\end{corollary}

This corollary is an immediate consequence of Theorem~\ref{sfc} and the relation
\[ \Lambda^{\operatorname{sf}}_t(N)\ll (\log N)\ \Gamma^{\operatorname{sf}}_t(N), \]
which was established in Section~4 of \cite{ABS}.
\section{Three lemmas on the completeness property}

The terminology introduced in Section~\ref{first} will now be transferred to our multi-index setting. We say that $B$ is divisor closed if $\beta\in B$ and $e_j\le \beta$ imply that $\beta-e_j\in B$. Moreover, if the assumption that $\beta\in B$ and $e_j\le \beta$ also implies that either $e_i\le \beta$ or $\beta-e_j+e_i\in B$ whenever $i<j$, then $B$ is said to have the completeness property. If $B$ is both divisor closed and enjoys the completeness property, then we say that $B$ is a complete set of multi-indices. We will say that a square-free set of distinct multi-indices $B=\{\beta_1,...,\beta_N\}$ is extremal with respect to $t$ if
\[ 
S(t,B)=N\ \Gamma^{\text{sf}}_t(N).\]
G\'{a}l's result \cite{G} showing the existence of divisor closed extremal sets with respect to an arbitrary sequence $t$ carries over to the multi-index setting without any change. This is the content of our first lemma:

\begin{lemma}\label{algorithm}
For every decreasing positive sequence $t$ in $\D^\infty$ and every positive integer $N$, there exists a divisor closed set $B$ of $N$ distinct square-free multi-indices that is extremal with respect to $t$.
\end{lemma}

\begin{proof}
We begin by assuming that we are given an arbitrary set $B$ of $N$ distinct square-free multi-indices. We fix an index $j$. Then for every $\beta$ in $B$ with the property that $\beta^{(j)}=1$ and $\beta-e_j$ is not in $B$, we replace $\beta$ by $\beta-e_j$. We call this new set $B$ as well and find by inspection that the value of $S(t,B)$ will be at least as large as the value we started with. Iterating this replacement for every index $j$ such that $\beta^{(j)}=1$ for some $\beta$ in $B$, we finally obtain a divisor closed set $B$ with $S(t,B)$ at least as large as the value we started with.

It follows that a set $B_0$ that maximizes $S(t,B)$ among all divisor closed sets $B$ of distinct square-free multi-indices, will automatically be extremal with respect to $t$. There can be at most $N-1$ indices $j$ such that $\beta^{(j)}=1$ for some $\beta$ in such a set $B$. Since $t$ is decreasing, we may assume that every $\beta$ in $B$ satisfies $\beta^{(j)}=0$ for $j\ge N$. This means that we are left with a finite dimensional problem and may conclude that the desired set $B_0$ exists.
\end{proof}

The preceding result implies that the assumption of the next lemma is not void.

\begin{lemma} \label{divisor}
Suppose that $B$ is a divisor closed set of $N$ distinct square-free multi-indices that is extremal with respect to some decreasing positive sequence $t$ in $\D^\infty$.  Then $B$ is a complete set of multi-indices.
%\[
%\beta'^{(\ell)}=\begin{cases} \beta^{(\ell}, & \text{if $\ell\neq j$ or $\beta^{(j)}\ge \beta^{(k)}  \\
%                                         then there exists a $\kappa(t)$-canonical set of $N$ multi-indices $B'=\{\beta'_1,...,\beta'_N\}$ such that
%$S(\eta( t),B')\ge S(t,B)$ and $\# \bigcup_{j=1}^N \supp \beta_j'\le N-1$.
\end{lemma}

\begin{proof}
We begin by assuming that there exists a multi-index $\beta$ in $B$ such that $\beta^{(j)}=1$ and $\beta^{(i)}=0$ for some $i<j$, but that 
$\beta-e_j+e_i$ is not in $B$. It suffices to show that this assumption implies that $B$ is not extremal with respect to any decreasing sequence $t$ of positive numbers in $\D^{\infty}$.

Set
\[ B_0:=\{ \mu\in B: \ \mu^{(j)}=1, \mu^{(i)}=0 \ \text{and} \ \mu-e_j+e_i\not\in B\}.\]
By our assumption, this set is non-empty. Now define 
\[ B_0':=\{\mu-e_j+e_i:\ \mu\in B_0\}\]
and correspondingly $B':=(B\setminus B_0)\bigcup B_0'$. Since $B'$ also has cardinality $N$, it suffices to show that
\begin{equation} \label{enough} S(t,B')>S(t,B) \end{equation}
holds for an arbitrary decreasing positive sequence $t$ in $\D^{\infty}$. Since
\[ \sum_{\mu, \nu \in B_0'} t^{|\mu-\nu|}=\sum_{\mu, \nu \in B_0} t^{|\mu-\nu|},\]
we have that \eqref{enough} is valid if and only if
\[ \sum_{\mu\in B_0', \nu\in B\setminus B_0} t^{|\mu-\nu|}>\sum_{\mu\in B_0, \nu\in B\setminus B_0} t^{|\mu-\nu|}. \]

We now split $B\setminus B_0$ into 4 disjoint subsets:
\begin{eqnarray*}
B_1 & := & \{ \nu \in B:\ \nu+(1-\nu^{(i)})e_i+(1-\nu^{(j)})e_j\in B\}; \\
B_2 & := & \{ \nu\in B\setminus B_1: \ \nu+(1-\nu^{(i)})e_i, \nu+(1-\nu^{(j)})e_j\in B\}; \\
B_3 & := & \{ \nu\in B\setminus (B_1\cup B_2): \ \nu+(1-\nu^{(i)})e_i \in B\}; \\
B_4 & := & B\setminus (B_0\cup B_1\cup B_2 \cup B_3).
\end{eqnarray*}
These sets are disjoint by construction, and we observe, using the assumption that $B$ is divisor closed, that each of the sets $B_1$, $B_2$, and $B_3$ has empty intersection with $B_0$. We may understand this splitting as follows: We look at every $\nu_0$ in $B$ with $\nu_0^{(i)}=\nu_0^{(j)}=0$ and decide first whether $\nu_0+e_i+e_j$ is in $B$, then---if not---whether both $\nu_0+e_i$ and $\nu_0+e_j$ are in $B$, and so on.  

% let $\nu$ be a multi-index in $B$ such that $\nu^{(j)}=\nu^{(i)}=0$. For each such $\nu$, there are four mutually exclusive
%possibilities for those elements $\nu'$ in $B\setminus B_0$ that satisfy $\nu\le \nu'$: 
%\begin{enumerate}
%\item  $\nu$, $\nu+e_j$, $\nu+e_i$, $\nu+e_j+e_i$ are all in $B\setminus B_0$
%\item  only $\nu$, $\nu+e_j$, $\nu+e_i$ are in $B\setminus B_0$
%\item only $\nu$ and $\nu+e_i$ are in $B\setminus B_0$
%\item only $\nu$ is in $B\setminus B_0$.
%\end{enumerate}
We note that $\beta-e_j$ is in $B_4$, whence $B_4$ is non-empty. We find that   \begin{eqnarray*} \sum_{\mu\in B_0', \nu\in B\setminus B_0} t^{|\mu-\nu|} &=&
\sum_{\mu\in B_0, \nu\in B_1}t^{|\mu-\nu|} +\frac{1+t_it_j+t_i}{1+t_i t_j + t_j}\sum_{\mu\in B_0, \nu\in B_2} t^{|\mu-\nu|} \\
& & +t_j^{-1}\sum_{\mu\in B_0, \nu\in B_3} t^{|\mu-\nu|}+t_j^{-1}t_i\sum_{\mu\in B_0, \nu\in B_4} t^{|\mu-\nu|}. \end{eqnarray*}
 Here the four coefficients in front of the sums are at least $1$. In addition, since 
 $B_4$ is non-empty and $t_j<t_i$, we conclude that \eqref{enough} holds and that $B$ is not extremal with respect to any decreasing positive sequence in $\D^{\infty}$.
\end{proof}

The preceding lemma is useful because of the following estimate.
\begin{lemma}
\label{l2}
Let $B$ be a complete set of $N$ distinct multi-indices, and let $\beta$ be a multi-index in $B$. If $\log N/\log 2\le j_1<j_2<\cdots <j_k$ and
$\beta^{(\ell)}\ge 1$ for $\ell=j_1, j_2,...,j_k$, then
\[ \sum_{i=1}^k \log j_i -k \log_2 N\le 3\log N.\]
\end{lemma}
\begin{proof}
Since $B$ is assumed to be complete, we must have that 
$ \sum_{i=1}^k e_{s_i}$ is in $B$
whenever the indices $s_1<\cdots<s_k$ satisfy $s_i\le j_i$ for $1\le i \le k$. It is clear that the total number of such elements in $B$ is at least
\[
\left(\prod_{i=1}^{k} (j_i-i+1)\right)/k!.
\]
Thus
\begin{equation}
\label{cp}
\left(\prod_{i=1}^{k} (j_i-i+1)\right)/k!\le N.
\end{equation}
Since $B$ is assumed to be divisor closed, we have $2^k\le N$ and hence $k\le \log N/\log 2$. Therefore, $k!\le (\log N/ \log 2)^k$. In addition, for each $i=1,\ldots,k$, we have that
\[ \frac{i-1}{j_{i}}\le \frac{i-1}{\log N/\log 2+i-1}\le 1/2,\]
since $i-1\le k\le\log N/\log 2$. Thus it follows from~\eqref{cp} that
$$
\left(\prod_{i=1}^{k} (j_i/2)\right)\cdot\left(\frac{\log N}{ \log 2}\right)^{-k}\le N.
$$
Taking logarithms on both sides of this inequality, we arrive at the statement of the lemma.
\end{proof}

\section{Proof of Theorem~\ref{sfc}}\label{proofsfc}

In \cite{ABS}, the starting point of the analysis was a rewriting of $S(t,B)$ as a certain Poisson integral formula. We will now replace this formula by  a somewhat simpler relation. To state the corresponding estimate, we introduce the following notation. We write $[\beta,\mu]:=(\max(\beta^{(k)},\mu^{(k)}))$ (the least common multiple of $\beta$ and $\mu$) and let $B^*$ be the set of all distinct multi-indices $[\beta_k,\beta_\ell]$ when $\beta_k$ and $\beta_{\ell}$ range over all multi-indices in $B$.
\begin{lemma}\label{sum} For arbitrary $t$ and $B$, we have
\[ S(t,B)\le  \sum_{\beta\in B^*} \left(\sum_{k: \beta_k\le \beta} t^{\beta-\beta_k}\right)^2.\]
\end{lemma}
\begin{proof}
We begin by observing that for every $\beta$ we have
\[  \left(\sum_{k: \beta_k\le \beta} t^{\beta-\beta_k}\right)^2=\sum_{k,\ell:\beta_k,\beta_\ell \le \beta} t^{2\beta-\beta_k-\beta_\ell}=\sum_{k,\ell:\beta_k,\beta_\ell \le \beta} t^{2{\beta-2[\beta_k,\beta_\ell]}} t^{|\beta_k-\beta_\ell|}.\]
Now summing over all $\beta$ in $B^*$, we obtain the desired result.
\end{proof}
 
In view of Lemma~\ref{sum} and Lemma~\ref{divisor}, it suffices to show that
\begin{equation}\label{toprove}
\sum_{\beta\in B^*} \left(\sum_{k: \beta_k\le \beta} t^{\beta-\beta_k}\right)^2\ll N\exp\left(C\sqrt{\frac{\log N\log_3 N}{\log_2N}}\right)
\end{equation}
holds whenever $B$ is a complete set of $N$ multi-indices.
We begin by fixing an auxiliary sequence $w=(w_j)$ as follows:
\[ w_j:=\begin{cases} t_j, & j\le \log N/\log 2\\
\sqrt{C/6}\sqrt{\frac{\log_3 N}{\log N \log_2 N}} (\log j -\log_2 N), & j>\log N/\log 2.
\end{cases}
\]
By the Cauchy--Schwarz inequality, we get  that
\begin{equation}
\label{CS}
 \left(\sum_{k:\ \beta_k \le \beta } t^{\beta-\beta_k} \right)^2  \le  \sum_{ k:\ \beta_k \le \beta } w^{\beta-\beta_k}
\sum_{ \ell:\ \beta_\ell \le \beta } w^{-(\beta-\beta_\ell)}t^{2(\beta-\beta_\ell)}.
\end{equation}
The first factor on the right-hand side of \eqref{CS} can be bounded by an Euler product:
\[
\sum_{ k:\ \beta_k \le \beta } w^{\beta-\beta_k}\le\prod_{j\in\supp\beta}(1+w_j).
\]
We now make the splitting $\supp\beta=I_1\bigcup I_2,$ where $I_1:=\supp \beta \cap \{j: \ j\le \log N/\log 2\}$ and $I_2=\supp \beta \setminus I_1$. Correspondingly, we get
\begin{eqnarray}
\prod_{j\in\supp\beta}(1+w_j)&=&\prod_{i\in I_1}(1+w_i) \prod_{j\in I_2}(1+w_j) \le \exp\left(\sum_{i=1}^{[\log N/\log 2]}t_i\right)\exp\left(\sum_{j\in I_2} w_j\right)\nonumber \\
\label{w1}
&\ll&\exp\left(C\sqrt{\frac{\log N}{\log_2 N}}\right)\exp\left(\sum_{j\in I_2} w_j\right),
\end{eqnarray}
where we, in the last step, used our assumption on $t$.
To estimate the second factor in~\eqref{w1}, we use that, for some $k$ and $\ell$, $\beta=[\beta_k,\beta_\ell]$. Therefore
\[
\sum_{j\in I_2}w_j \le\sum_{j\in I_2\cap\supp\beta_k}w_j+\sum_{j\in I_2\cap\supp\beta_{\ell}}w_j.
\]
By the definition of $w_j$, we have, for every $m=1,\ldots,N$,
\begin{eqnarray*}
\sum_{j\in I_2\cap\supp\beta_m}w_j&\le&\sqrt{C/6}\sqrt{\frac{\log_3 N}{\log N \log_2 N}} \sum_{j:\,j>\log N/\log 2,\, j\in\supp\beta_m}(\log j -\log_2 N)\\
&\le&3\sqrt{C/6}\log N\sqrt{\frac{\log_3 N}{\log N \log_2 N}}=3\sqrt{C/6}\sqrt{\frac{\log N \log_3 N}{\log_2 N}},
\end{eqnarray*}
where, in the last inequality, we used Lemma~\ref{l2}. Hence
\[
\exp\left(\sum_{j\in I_2}w_j\right)\le \exp\left(\sqrt{6C}\sqrt{\frac{\log N \log_3 N}{\log_2 N}}\right).
\]
This inequality together with~\eqref{w1} implies that, for every $\beta$ in $B^*$, we have
\be
\label{w}
\sum_{ k:\ \beta_k \le \beta } w^{\beta-\beta_k}\le\exp\left((\sqrt{6C}+o(1))\sqrt{\frac{\log N \log_3 N}{\log_2 N}}\right).
\ee

By \eqref{CS} and \eqref{w}, the left-hand side of \eqref{toprove} can be estimated as 
\begin{equation}\label{inter}
\sum_{\beta\in B^*} \left(\sum_{k: \beta_k\le \beta} t^{\beta-\beta_k}\right)^2
\le \exp\left((\sqrt{6C}+o(1))\sqrt{\frac{\log N \log_3 N}{\log_2 N}}\right) \sum_{\beta\in B^*}
\sum_{ k:\ \beta_k \le \beta } w^{-(\beta-\beta_k)}t^{2(\beta-\beta_k)},
\end{equation}
and so we are left with the problem of estimating the second factor on the right-hand side of \eqref{inter}. To deal with this double sum, we begin
by changing the order of summation:
\[
\sum_{\beta\in B^*}
\sum_{ k:\ \beta_k \le \beta } w^{-(\beta-\beta_k)}t^{2(\beta-\beta_k)}=\sum_{k=1}^N\sum_{\beta\in B^*:\, \beta\ge\beta_k}w^{-(\beta-\beta_k)}t^{2(\beta-\beta_k)}.
\]
Thus it is enough to show that, for every $k=1,\ldots,N$, we have
\be
\label{sum2}
\sum_{\beta\in B^*:\, \beta\ge\beta_k}w^{-(\beta-\beta_k)}t^{2(\beta-\beta_k)}\ll\exp\left((\sqrt{6C}+o(1))\sqrt{\frac{\log N \log_3 N}{\log_2 N}}\right).
\ee
Set $J:=\bigcup_{j=1}^N\{\supp\beta_j\}$. For a fixed $k$, we have, resorting again to an Euler product: 
$$
\sum_{\beta\in B^*:\, \beta\ge\beta_k}w^{-(\beta-\beta_k)}t^{2(\beta-\beta_k)}\le\prod_{j\in J}\left(1+\frac{t_j^2}{w_j}\right).
$$
Now we make the splitting $J=J_1\bigcup J_2$, where $J_1:=J\cap \{j:\ j \le \log N/\log 2\}$ and $J_2:=J\setminus J_1$. Hence we get
\be \label{prod1}
\prod_{j\in J}\left(1+\frac{t_j^2}{w_j}\right)=\prod_{i\in J_1}\left(1+\frac{t_i^2}{w_i}\right)\prod_{j\in J_2}\left(1+\frac{t_j^2}{w_j}\right)
\le \exp\left(\sum_{i=1}^{[\log N/\log 2]}\frac{t_i^2}{w_i}\right)\exp\left(\sum_{j\in J_2}\frac{t_j^2}{w_j}\right).
\ee
Since $t_i=w_i$ for $i\le\log N/\log 2$,  we obtain
\be
\label{prod2}
\exp\left(\sum_{i=1}^{[\log N/\log 2]}\frac{t_i^2}{w_i}\right)\ll\exp\left(C\sqrt{\frac{\log N}{\log_2 N}}\right)
\ee
by our assumption on $t$. Therefore, to prove \eqref{sum2}, it is enough to show that
\be
\label{prod3}
\sum_{j\in J_2}\frac{t_j^2}{w_j}\le (\sqrt{6C}+o(1))\sqrt{\frac{\log N \log_3 N}{\log_2 N}}.
\ee
For every sufficiently large $j$ in $J_2$, we have 
$$
\frac{t_j^2}{w_j}\le \sqrt{\frac{\log N \log_2 N}{\log_3 N}} \frac{\sqrt{6C}}{j\log j(\log j -\log_2 N)}.
$$
Thus, to finish the proof of \eqref{prod3}, we need to estimate the sum
$$
S:=\sum_{j=[\log N/\log 2]+1}^{\infty}\frac 1{j\log j(\log j -\log_2 N)}.
$$
Since the function
$$
g(x):=\frac 1{x\log x(\log x -\log_2 N)}
$$
is decreasing for $x>\log N$, we  obtain that
\begin{eqnarray*}
S&\le&\int_{\log N/\log 2}^{\infty}\frac{dx}{x\log x(\log x -\log_2 N)}=\int_{\log_2 N-\log_2 2}^{\infty}\frac{du}{u(u-\log_2 N)} \\
&=&\frac 1{\log_2 N}\log\left(\frac{\log_2 N-\log_2 2}{-\log_2 2}\right)= \frac{\log_3 N}{\log_2 N}+O(1/\log_2 N).
\end{eqnarray*}
It follows that \eqref{prod3} and hence~\eqref{sum2} hold.

We finally notice that we can choose $\kappa$ in \eqref{est2} less than $5$ because $2\sqrt{6}<5$.  

\section{Proof of Theorem~\ref{maintheorem}}\label{proof}

As in \cite{ABS}, we want to transfer our analysis from a given sequence $t=(t_j)$ to the sequence in which $t_j$ is replaced by $2t_j$.  However, since we must require the new sequence to be contained in $\D^{\infty}$, a slight adjustment is needed. To this end, we introduce the function $\eta:(0,1)\to (0,1)$  defined by the relation
\[  \eta(x):=\begin{cases} 2x, & 0<x <1/2 \\
x, & 1/2 \le x <1. \end{cases} \]
Then, for a sequence $t=(t_j)$ with $0<t_j<1$, we set $\eta(t):=(\eta(t_j))$. In the case when $t_j \searrow 0$, we let $\omega(t)$ denote the number of $t_j$ such that $t_j>1/2$. In the concrete case to be considered below, when $t_j=1/\sqrt{p_j}$, we get
$\eta(t)=(1/\sqrt{2}, 1/\sqrt{3}, 2/\sqrt{5}, 2/\sqrt{7},...)$ and $\omega(t)=2$. 

We will use the following variant of \cite[Lemma 2]{ABS}. 
\begin{lemma}
\label{closedset}
Suppose $B$ is a set of $N$ multi-indices. Let $t$ be a decreasing sequence of positive numbers in $\D^{\infty}\cap c_0$. Then there exists a divisor closed set of $N$ multi-indices $B'=\{\beta'_1,...,\beta'_N\}$ such that
$S(t,B)\le 2^{\omega(t)}S(\eta(t),B')$.  
\end{lemma}
The proof is essentially identical to the proof of  \cite[Lemma 2]{ABS}, the only minor difference being that we also involve the $t_j$ satisfying $1/2\le t_j<1$. This is why we get an extra factor $2^{\omega(t)}$. We refer to \cite{ABS} for the details.

\begin{proof}[Proof of Theorem~\ref{maintheorem}]
We need to estimate $\Gamma_t(N)$ in the particular case when 
$t=(p_j^{-1/2})$. By Lemma~\ref{closedset}, we have that 
\be \label{Gd} \Gamma_t(N)\le 4 N^{-1} \sup_{B \ \text{divisor closed}} S(\eta(t),B), \ee
where the supremum is taken over all divisor closed sets of multi-indices $B=(\beta_1,\ldots,\beta_N)$. We will estimate the right-hand side of \eqref{Gd}. To this end, we introduce for convenience the notation $u:=\eta(t)$. By an obvious rearrangement, we may assume that the sequence $u=(u_j)$ is decreasing; we have then $u_j=2p_j^{-1/2}$ for $j\ge 6$.

For every square-free multi-index $\beta$ in $B$, we define $B_\beta$ to be the collection of all multi-indices $\beta'$ in $B$ satisfying
$\supp \beta' =\supp \beta$. Assuming for convenience that $\beta_1,..., \beta_M$ ($M\le N$) are the square-free multi-indices in $B$, we then have 
\[ B=\bigcup_{k=1}^M B_{\beta_k} \ \ \text{and} \ \   \sum_{k=1}^M |B_{\beta_k}|=N.\]
Now suppose that $k, \ell \le M$ and that $|B_{\beta_k}|\le |B_{\beta_\ell}|$. We then find that
\[
\sum_{\beta \in B_{\beta_k}} \sum_{\beta'\in B_{\beta_\ell}} u^{|\beta-\beta'|}\le u^{|\beta_k-\beta_\ell|}
\sum_{\beta\in B_{\beta_k}} 2^6\prod_{i \in \supp\beta_\ell, i\le 6}(1-u_i)^{-1}\prod_{j \in \supp\beta_\ell, j>6}(1-2u_j)^{-1} .
\]
Here we used the fact that $u_{i}<1/2$ for $i>6$ and that any given value for $|\beta-\beta'|$ can be obtained by $2^{m}$ different
multi-indices $\beta'$, when $\beta$ is fixed; here $m$ denotes the value of $|\supp\beta'|$ which must be the same for all these $\beta'$.
Since $B'$ is divisor closed, we have $|\supp{\beta_\ell}|\le\log N/\log 2$, and we get
\be
\label{mu1}
\sum_{\beta \in B_{\beta_k}} \sum_{\beta'\in B_{\beta_\ell}} u^{|\beta-\beta'|}\le u^{|\beta_j-\beta_k|} 2^6 |B_{\beta_k}|\exp\left(K\sqrt{\frac{\log N}{\log_2 N}}\right) 
\ee
for an absolute constant $K$.
Summing \eqref{mu1} over $k$ and $\ell$ and recalling that $|B_{\beta_k}|\le |B_{\beta_\ell}|$, we obtain
\[ S(u,B)\le 2^6 \exp\left(K\sqrt{\frac{\log N}{\log_2 N}}\right) \sum_{k,\ell =1}^M |B_{\beta_k}|^{1/2}|B_{\beta_\ell}|^{1/2}u^{|\beta_k-\beta_\ell|}.\]
Now Corollary~\ref{sfceigen} applies because the set $\{\beta_1,...,\beta_M\}$ is square-free. Using also the prime number theorem, we therefore
get 
\begin{eqnarray*}  S(u,B)&\le & 2^6 \exp\left(K\sqrt{\frac{\log N}{\log_2 N}}\right) \Lambda^{\text{sf}}_{u}(M) \sum_{k=1}^M |B_{\beta_k}| \\
 & \le & 2^6 C_{\varepsilon} N\exp\left(K\sqrt{\frac{\log N}{\log_2 N}}+(\sqrt{2}+\varepsilon)\kappa\sqrt{\frac{\log N\log_3N}{\log_2N}}\right) 
\end{eqnarray*}
for every positive $\varepsilon$. Returning to \eqref{Gd} and recalling that $\kappa$ can be chosen as any number larger than $2\sqrt{6}$,
we finally obtain
\[  \Gamma_t(N) \ll \exp\left(A\sqrt{\frac{\log N\log_3N}{\log_2N}}\right) \]
with $A<7$ because $2\sqrt{12}<7$. 
\end{proof}

%\end{proof}

\section{Concluding remarks}\label{concluding}

\noindent 1. We may deduce from Theorem~\ref{maintheorem} an upper bound for the spectral norms of the GCD matrices $(t^{|\beta_j-\beta_{\ell}|})$. To this end, we define  
\be \label{lam} \Lambda_t(N):=\sup_{\beta_1,...,\beta_N} \sup_{c\neq 0} \frac{\sum_{k,\ell=1}^N t^{|\beta_k-\beta_{\ell}|} c_k \overline{c}_{\ell}}{\sum_{j=1}^N |c_j|^2}, \ee
where the suprema are taken over respectively all $N$-tuples of distinct multi-indices $\beta_1,...\beta_N$ and all nonzero vectors $c=(c_1,...,c_N)$ in $\C^N$.   Following Section~4 of \cite{ABS}, we obtain 
\be \label{our} \Lambda_{(p_j^{-1/2})}(N)\ll \exp\left(A\sqrt{\frac{\log N \log_3 N}{\log_2 N}}\right) \ee
for a constant $A$ less than 7.  Replacing Theorem 3.2 of \cite[p. 62]{H} by \eqref{our}, we obtain corresponding improvements of the results  in \cite[Chapter 3]{H} depending on that theorem. See also Hilberdink's paper \cite{Hi}, as mentioned in the introduction.\vspace{2mm}

\noindent 2. Note that the GCD sums seem indispensable for estimating $\Lambda_t(N)$ because of the central role played by the completeness property. Indeed, for the proof of Lemma~\ref{divisor} to work, it is of crucial importance that $c_k\equiv C$ in the quadratic sum on the right-hand side of \eqref{lam}. 
\vspace{2mm}
 
\noindent 3. The method of this paper applies without any essential change to the somewhat easier case $1/2 < \alpha < 1$.  It gives an improvement of the constants
$A_{\alpha}$ in the bound
\[ \Gamma_{(p_j^{-\alpha})}(N)\ll \exp\left(A_{\alpha} \frac{(\log N)^{1-\alpha}}{(\log_2 N)^{\alpha}}\right)\]
from \cite{ABS}. This estimate is optimal up to the precise value of $A_{\alpha}$. However, our method of proof does not apply to the case $0<\alpha <1/2$, except that  Theorem~\ref{maintheorem} along with H\"{o}lder's inequality yields the estimate
\[ \Gamma_{(p_j^{-\alpha})}(N)\ll N^{1-2\alpha} \exp\left(\alpha A \sqrt{\frac{\log N \log_3 N}{\log_2 N}}\right),\]
cf. the proof of \cite[Theorem~1]{ABS}.  It is tempting to conjecture that the second factor in this bound is an artifact and that the bound should be
just $N^{1-2\alpha}$ or at least much closer to $N^{1-2\alpha}$. Our two lemmas on the completeness property still apply, but otherwise new ideas seem to be needed to handle this case. Note that now we know the precise value of the leading term in the expression for $\log \Gamma_t(N)$:
\[ \log \Gamma_{(p_j^{-\alpha})}(N)=(1-2\alpha)\log N + \text{lower order term}. \]
The question is then if we are able to find the right order of growth of the lower order term. This is a problem of a rather different nature than that of finding the right order of growth of the main term when $1/2\le \alpha <1$.
\vspace{2mm}

\noindent 4. It remains an interesting problem to find the exact growth rate of $\log \Gamma_{(p_j^{-1/2})}(N)$.  From \cite{ABS} we know that we reach the lower bound in \eqref{lower} if we choose the set of cardinality $2^k$ consisting of all square-free numbers generated by the first $k$ primes.  We can not rule out the possibility that a more sophisticated example, still having the completeness property, could  be worked out to give a larger lower bound.

\section*{Acknowledgements} The authors are grateful to Christoph Aistleitner and Istv\'{a}n Berkes for pertinent and helpful remarks on the bibliography of this paper. They would also like to thank the anonymous referee for a careful review and for suggesting some essential clarifications of the exposition.

\end{document}